\documentclass[letterpaper,10pt]{amsart}
\usepackage{amsmath,amssymb,amsxtra,amsthm,fullpage}

\makeatletter
\@namedef{subjclassname@2010}{%
  \textup{2010} Mathematics Subject Classification}
\makeatother

\newtheorem{theorem}{Theorem}[section]

\theoremstyle{definition}
\newtheorem{definition}[theorem]{Definition}

\numberwithin{equation}{section}

\begin{document}

\title{Iterative Methods for Computing Eigenvalues and Eigenvectors}
\author{Maysum Panju}
\address{University of Waterloo, Faculty of Mathematics, Waterloo, ON, N2L 3G1, Canada}
\email{mhpanju@math.uwaterloo.ca}
\date{\today}


\begin{abstract} We examine some numerical iterative methods for computing the eigenvalues and eigenvectors of real matrices. The five methods examined here range from the simple power iteration method to the more complicated QR iteration method. The derivations, procedure, and advantages of each method are briefly discussed.
\end{abstract}

\maketitle

\section{Introduction}

Eigenvalues and eigenvectors play an important part in the applications of linear algebra. The naive method of finding the eigenvalues of a matrix involves finding the roots of the characteristic polynomial of the matrix. In industrial sized matrices, however, this method is not feasible, and the eigenvalues must be obtained by other means. Fortunately, there exist several other techniques for finding eigenvalues and eigenvectors of a matrix, some of which fall under the realm of iterative methods. These methods work by repeatedly refining approximations to the eigenvectors or eigenvalues, and can be terminated whenever the approximations reach a suitable degree of accuracy. Iterative methods form the basis of much of modern day eigenvalue computation. 

In this paper, we outline five such iterative methods, and summarize their derivations, procedures, and advantages. The methods to be examined are the power iteration method, the shifted inverse iteration method, the Rayleigh quotient method, the simultaneous iteration  method, and the QR method. This paper is meant to be a survey over existing algorithms for the eigenvalue computation problem.

Section 2 of this paper provides a brief review of some of the linear algebra background required to understand the concepts that are discussed. In section 3, the iterative methods are each presented, in order of complexity, and are studied in brief detail. Finally, in section 4, we provide some concluding remarks and mention some of the additional algorithm refinements that are used in practice.

For the purposes of this paper, we restrict our attention to real-valued, square matrices with a full set of real eigenvalues.

\section{Linear Algebra Review}

We begin by reviewing some basic definitions from linear algebra. It is assumed that the reader is comfortable with the notions of matrix and vector multiplication.

\begin{definition} 
Let $A \in \mathbb{R}^{n \times n}$. A nonzero vector $x \in \mathbb{R}^n$ is called an \emph{eigenvector} of $A$ with corresponding \emph{eigenvalue} $\lambda \in \mathbb{C}$ if $Ax = \lambda x$.
\end{definition}

Note that eigenvectors of a matrix are precisely the vectors in $\mathbb{R}^n$ whose direction is preserved when multiplied with the matrix. Although eigenvalues may be not be real in general, we will focus on matrices whose eigenvalues are all real numbers. This is true in particular if the matrix is symmetric; some of the methods we detail below only work for symmetric matrices.

It is often necessary to compute the eigenvalues of a matrix. The most immediate method for doing so involves finding the roots of characteristic polynomials.

\begin{definition}
The \emph{characteristic polynomial} of $A$, denoted $P_A(x)$ for $x \in \mathbb{R}$, is the degree $n$ polynomial defined by $$P_A(x) := \det(xI - A).$$
\end{definition}

It is straightforward to see that the roots of the characteristic polynomial of a matrix are exactly the eigenvalues of the matrix, since the matrix $zI - A$ is singular precisely when $z$ is an eigenvalue of $A$. It follows that computation of eigenvalues can be reduced to finding the roots of polynomials. Unfortunately, solving polynomials is generally a difficult problem, as there is no closed formula for solving polynomial equations of degree 5 or higher. The only way to proceed is to employ numerical techniques to solve these equations.

We have just seen that eigenvalues may be found by solving polynomial equations. The converse is also true. Given any monic polynomial $$f(z) = z^n + a_{n-1}z^{n-1} + \ldots + a_1z + a_0,$$ we can construct the companion matrix
\[ \left[ \begin{array}{ccccccc}
0 &   &  &  \cdots    &   & -a_0   \\
1 & 0 &   &           &   & -a_1   \\
  & 1 & 0 &           &   & -a_2   \\
  &   &  & \ddots     &   & \vdots \\   
  &   &   &         1 & 0 & -a_{n-2} \\
  &   &   &           & 1 & -a_{n-1} \end{array} \right]\]

It can be seen that the characteristic polynomial for the companion matrix is exactly the polynomial $f(z)$. Thus the problem of computing the roots of a polynomial equation reduces to finding the eigenvalues of a corresponding matrix. Since polynomials in general cannot be solved exactly, it follows that there is no method that will produce exact eigenvalues for a general matrix.

However, there do exist methods for computing eigenvalues and eigenvectors that do not rely upon solving the characteristic polynomial. In this paper, we look at some iterative techniques used for tackling this problem. These are methods that, when given some initial approximations, produce sequences of scalars or vectors that converge towards the desired eigenvalues or eigenvectors. On the other hand, we can make the notion of convergence of matrices precise as follows.

\begin{definition}
Let $A^{(1)}, A^{(2)}, A^{(3)}, \ldots$ be a sequence of matrices in $\mathbb{R}^{m \times n}$. We say that the sequence of matrices \emph{converges} to a matrix $A \in \mathbb{R}^{m \times n}$ if the sequence $A^{(k)}_{i,j}$ of real numbers converges to $A_{i,j}$ for every pair $1 \leq i \leq m$, $1 \leq j \leq n$, as $k$ approaches infinity. That is, a sequence of matrices converges if the sequences given by each entry of the matrix all converge.
\end{definition}

Later in this paper, it will be necessary to use what is known as the QR decomposition of a matrix. 

\begin{definition}
The \emph{QR decomposition} of a matrix $A$ is the representation of $A$ as a product $$A = QR,$$ where $Q$ is an orthogonal matrix and $R$ is an upper triangular matrix with positive diagonal entries.
\end{definition}

Recall that an \emph{orthogonal} matrix $U$ satisfies $U^TU=I$. Importantly, the columns of $Q$ are orthogonal vectors, and span the same space as the columns of $A$. It is a fact that any matrix $A$ has a QR decomposition $A=QR$, which is unique when $A$ has full rank.

Geometrically, the QR factorization means that if the columns of $A$ form the basis of a vector space, then there is an orthonormal basis for that vector space. This orthonormal basis would would form the columns of $Q$, and the conversion matrix for this change of basis is the upper triangular matrix $R$. The methods for obtaining a QR decomposition of a matrix has been well studied and is a computationally feasible task\footnote{The simplest method for computing a QR factorization of a matrix $A$ is to apply the Gram-Schmidt algorithm on the columns of $A$. A student of linear algebra may be horrified at the prospect of carrying out this tedious procedure once, let alone once \emph{per iteration} of an iterative method, but recall that when working with matrices of thousands of rows or more, all computations are done electronically. Furthermore, the methods used to calculate a QR decomposition are usually more complicated than Gram-Schmidt; for instance, a common technique is to apply a series of Householder transformations \cite{George1987223}.}.

At this point, we turn our attention to the iterative methods themselves.
\section{Description of the Iterative Methods}

The iterative methods in this section work by repeatedly refining estimates of eigenvalues of a matrix, using a function called the Rayleigh quotient.

\begin{definition}
Let $A \in \mathbb{R}^{n \times n}$. Then the \emph{Rayleigh quotient} of a nonzero vector $x \in \mathbb{R}^n$ is $$r(x) := \frac{x^TAx}{x^Tx}.$$
\end{definition}

Note that if $x$ is an eigenvector for $A$ with corresponding eigenvalue $\lambda$, then the Rayleigh quotient for $x$ is
$$r(x) = \frac{x^TAx}{x^Tx} = \frac{\lambda x^Tx}{x^Tx} = \lambda$$
which is exactly the corresponding eigenvalue.

In fact, given any nonzero $x \in \mathbb{R}^{n}$, the Rayleigh quotient $r(x)$ is the value that minimizes the function $f(\alpha) = \left\|\alpha x - Ax\right\|_2$ over all real numbers $\alpha$, which measures the error incurred if $x$ is assumed to be an eigenvector of $A$. Thus, if we treat $x$ as an estimate for an eigenvector of $A$, then $r(x)$ can be seen as the best estimate for the corresponding eigenvalue of $A$, since it minimizes this error value.

We are now ready to consider the first technique for iterative eigenvalue computation.

\subsection{Power Iteration}

Let $A \in \mathbb{R}^{n \times n}$. Recall that if $q$ is an eigenvector for $A$ with eigenvalue $\lambda$, then $Aq = \lambda q$, and in general, $A^kq = \lambda^kq$ for all $k \in \mathbb{N}$. This observation is the foundation of the power iteration method.

Suppose that the set $\{q_i\}$ of unit eigenvectors of $A$ forms a basis of $\mathbb{R}^{n}$, and has corresponding real eigenvalues $\{\lambda_i\}$ such that $\left|\lambda_1\right| > \left|\lambda_2\right| > \ldots > \left|\lambda_n\right|$.  Let $v^{(0)}$ be an approximation to an eigenvector of $A$, with $\left\|v^{(0)}\right\| = 1$. (We use the term ``approximation'' in this situation quite loosely, allowing it to refer to any quantity that is believed to be ``reasonably close'' to the correct value\footnote{These approximations are always made when the correct value is unknown (indeed, they are made in an attempt to \emph{determine} the correct value); we therefore do not require any particularly demanding conditions on how ``close'' the approximation must be to being correct.}.)We can write $v^{(0)}$ as a linear combination of the eigenvectors of $A$; for some $c_1, \ldots , c_n \in \mathbb{R}$ we have that $$v^{(0)} = c_1q_1 + \ldots + c_nq_n,$$ and we will assume for now that $c_1 \neq 0$.

Now
$$Av^{(0)} = c_1\lambda_1 q_1 + c_2\lambda_2 q_2 +\ldots + c_n\lambda_n q_n$$
and so
\begin{eqnarray*}
	A^k v^{(0)} & = & c_1\lambda_1^kq_1 + c_2\lambda_2^k q_2 +\ldots + c_n\lambda_n^kq_n  \\
	& = &  \lambda_1^k\left(c_1q_1 + c_2\left( \frac{\lambda_2}{\lambda_1}\right)^kq_2 +\ldots + c_n\left( \frac{\lambda_n}{\lambda_1}\right)^kq_n\right) \\
\end{eqnarray*}

Since the eigenvalues are assumed to be real, distinct, and ordered by decreasing magnitude, it follows that for all $i = 2, \ldots, n$, 
$$\lim_{k\rightarrow \infty}\left(\frac{\lambda_i}{\lambda_1}\right)^k =0.$$
So, as $k$ increases, $A^kv^{(0)}$ approaches $c_1\lambda_1^k q_1$, and thus for large values of $k$,
$$q_1 \approx \frac{A^k v^{(0)}}{\left\|A^k v^{(0)}\right\|}.$$

The method of power iteration can then be stated as follows: \\ \\
	\texttt{\indent Pick a starting vector $v^{(0)}$ with $\left\|v^{(0)}\right\|=1$} \\
	\texttt{\indent For $k =1, 2, \ldots $ \\
	\texttt{\indent \indent Let $w = Av^{(k-1)}$} \\
	\texttt{\indent \indent Let $v^{(k)} = w/\left\|w\right\|$} }\\

In each iteration, $v^{(k)}$ gets closer and closer to the eigenvector $q_1$. The algorithm may be terminated at any point with a reasonable approximation to the eigenvector; the eigenvalue estimate can be found by applying the Rayleigh quotient to the resulting $v^{(k)}$.

The power iteration method is simple and elegant, but suffers some major drawbacks. The method only returns a single eigenvector estimate, and it is always the one corresponding to the eigenvalue of largest magnitude. In addition, convergence is only guaranteed if the eigenvalues are distinct---in particular, the two eigenvalues of largest absolute value must have distinct magnitudes. The rate of convergence primarily depends upon the ratio of these magnitudes, so if the two largest eigenvalues have similar sizes, then the convergence will be slow. 

In spite of its drawbacks, the power method is still used in some applications, since it works well on large, sparse matrices when only a single eigenvector is needed. However, there are other methods that overcome the difficulties of the power iteration method.

\subsection{Inverse Iteration}

The inverse iteration method is a natural generalization of the power iteration method. 

If $A$ is an invertible matrix with real, nonzero eigenvalues $\{\lambda_1, \ldots, \lambda_n\}$, then the eigenvalues of $A^{-1}$ are $\{1/\lambda_1, \ldots, 1/\lambda_n\}$. Thus if $\left|\lambda_1\right| > \left|\lambda_2\right| > \ldots > \left|\lambda_n\right|$, then $\left|1/\lambda_1\right| < \left|1/\lambda_2\right| < \ldots < \left|1/\lambda_n\right|$, and so by applying the power method iteration on $A^{-1}$, we can obtain the eigenvector $q_n$ and eigenvalue $\lambda_n$.

This gives a way to find the eigenvalue of smallest magnitude, assuming that $A^{-1}$ is known. In general, though, the inverse matrix is not given, and calculating it is a computationally expensive operation. However, computing $x = A^{-1}b$ is equivalent to solving the system $Ax = b$ for $x$ given $b$, and this operation can be efficiently performed. Fortunately, this is all that is required for the inverse iteration method, which we can now state as follows: \\ \\
  \texttt{\indent Pick a starting vector $v^{(0)}$ with $\left\|v^{(0)}\right\|=1$} \\
	\texttt{\indent For $k =1, 2, \ldots $ \\
	\texttt{\indent \indent Solve $Aw = v^{(k-1)}$ for $w$} \\
	\texttt{\indent \indent Let $v^{(k)} = w/\left\|w\right\|$} }\\

The advantage of inverse iteration is that it can be easily adapted to find \emph{any} eigenvalue of the matrix $A$, instead of just the extreme ones. Observe that for any $\mu \in \mathbb{R}$, the matrix $B = A - \mu I$ has eigenvalues $\{\lambda_1 - \mu, \ldots, \lambda_n - \mu\}$. In particular, by choosing $\mu$ to be close to an eigenvalue $\lambda_j$ of $A$, we can ensure that $\lambda_j - \mu$ is the eigenvalue of $B$ of smallest magnitude. Then by applying inverse iteration on $B$, an approximation to $q_j$ and $\lambda_j$ can be obtained.

This version of the algorithm, known as inverse iteration with shift, can be summarized as follows: \\ \\
  \texttt{\indent Pick some $\mu$ close to the desired eigenvalue} \\
  \texttt{\indent Pick a starting vector $v^{(0)}$ with $\left\|v^{(0)}\right\|=1$} \\
	\texttt{\indent For $k =1, 2, \ldots $ \\
	\texttt{\indent \indent Solve $(A-\mu I)w = v^{(k-1)}$ for $w$} \\
	\texttt{\indent \indent Let $v^{(k)} = w/\left\|w\right\|$} }\\

The inverse iteration with shift method allows the computation of any eigenvalue of the matrix. However, in order to compute a particular eigenvalue, you must have some initial approximation of it to start the iteration. In cases where an eigenvalue estimate is given, the inverse iteration with shift method works well.

\subsection{Rayleigh Quotient Iteration}

The inverse iteration method can be improved if we drop the restriction that the shift value remains constant throughout the iterations. 

Each iteration of the shifted inverse iteration method returns an approximate eigenvector, given an estimate of an eigenvalue. The Rayleigh quotient, on the other hand, produces an approximate eigenvalue when given an estimated eigenvector. By combining these two operations, we get a new variation of the inverse shift algorithm, where the shift value is recomputed during each iteration to become the Rayleigh quotient of the current eigenvector estimate.

This method, called the Rayleigh quotient iteration method (or simply the RQI method), is as follows: \\ \\  
  \texttt{\indent Pick a starting vector $v^{(0)}$ with $\left\|v^{(0)}\right\|=1$} \\
  \texttt{\indent Let $\lambda^{(0)} = r(v^{(0)}) := (v^{(0)})^TA(v^{(0)})$} \\
	\texttt{\indent For $k =1, 2, \ldots $ \\
	\texttt{\indent \indent Solve $(A-\lambda^{(k-1)}I)w = v^{(k-1)}$ for $w$} \\
	\texttt{\indent \indent Let $v^{(k)} = w/\left\|w\right\|$} }\\
	\texttt{\indent \indent Let $\lambda^{(k)} = r(v^{(k)}) := (v^{(k)})^TA(v^{(k)})$ } \\

In this method, we no longer need to have an initial eigenvalue estimate supplied; all that is required is an initial vector $v^{(0)}$. The eigenvector produced depends on the initial vector chosen. Note that since each vector $v^{(k)}$ is a unit vector, we have $\left(v^{(k)}\right)^T v^{(k)} = 1$, simplifying the expression for $r\left(v^{(k)}\right)$.

The main advantage of the RQI method is that it converges to an eigenvector very quickly, since the approximations to both the eigenvalue and the eigenvector are improved during each iteration. Thus far, we have not given any quantitative discussion on the speed at which an iteration sequence converges to the limit, since these technical details are not the purpose of this paper. However, it is worth mentioning here that the convergence rate of the RQI method is said to be \emph{cubic}, which means that the number of correct digits in the approximation triples during each iteration \cite{RayleighConv}. In contrast, the other algorithms described in this paper all have the much slower \emph{linear} rate of convergence.\footnote{In further contrast, the well known Newton's method for rapidly finding roots of differentiable functions has \emph{quadratic} convergence. Generally, the study of algorithm convergence rates is extremely technical, and the author feels no regret in omitting further details from this paper.}

One very significant disadvantage for the RQI method is that it does not always work in the general case. The method is only guaranteed to converge when the matrix $A$ is both real and symmetric, and is known to fail in the cases where the matrix is not symmetric.

\subsection{Simultaneous Iteration}

The methods discussed so far are only capable of computing a single eigenvalue at a time. In order to compute different eigenvalues of a matrix, the methods must be reapplied  several times, each time with different initial conditions. We now describe a method that is capable, in some situations, of producing all of the eigenvalues of a matrix at once.

Once again, our starting point is the basic power iteration method. Let $A$ be a real, symmetric, full rank matrix; in particular, $A$ has real eigenvalues and a complete set of orthogonal eigenvectors. Recall that given the starting vector $v^{(0)}$, we can write it as a linear combination of eigenvectors $\{q_i\}$ of $A$: 
$$v^{(0)} = c_1q_1 + \ldots + c_nq_n.$$
Note, however, that the power iteration method only looks at eigenvectors that are nontrivial components in this linear combination; that is, only the eigenvectors that are not orthogonal to $v^{(0)}$ have a chance of being found by the power iteration method. This suggests that by applying the power iteration method to several different starting vectors, each orthogonal to all of the others, there is a possibility of finding different eigenvalues.

With this idea in mind, we may take the following approach. Begin with a basis of $n$ linearly independent vectors $\{v_1^{(0)} \ldots v_n^{(0)}\}$ of $\mathbb{R}^n$, arranged in the matrix $$V^{(0)} = \left[v_1^{(0)}\right|\cdots \left| v_n^{(0)} \right].$$ Let $$V^{(k)} = A^kV^{(0)} = \left[v_1^{(k)}\right|\cdots \left| v_n^{(k)} \right],$$ which effectively applies the power iteration method to all of the vectors $\{v_1^{(0)} \ldots v_n^{(0)}\}$ at once. We thus expect that as $k \rightarrow \infty$, the columns of $V^{(k)}$ become scaled copies of $q_1$, the unit eigenvector corresponding to the eigenvalue of largest magnitude. (Note that the columns are not necessarily unit vectors themselves, since we did not normalize the vectors in this version of the algorithm).

So far, we have not found anything useful or new. We have obtained $n$ eigenvectors, but they are possibly all in the same direction. The main development occurs when we decide to orthonormalize the columns of $V^{(k)}$ at each iteration. In the original power iteration method, the obtained eigenvector estimate was normalized in each iteration. In this multivector version, the analogue is to obtain an orthonormal set of eigenvector estimates during each iteration, forcing the eigenvector approximations to be orthogonal at all times. This is done by using the QR decomposition of $V^{(k)}$.

In each step of this iteration method, we obtain a new matrix, $W$, by multiplying $A$ by the current eigenvector approximation matrix, $V^{(k-1)}$. We can then extract the orthonormal column vectors of $Q$ from the QR decomposition of $W$, thus ensuring that the eigenvector approximations remain an orthogonal basis of unit vectors. This process is then repeated as desired. The algorithm, known as the simultaneous iteration method \cite{simultIter}, can be written as follows: \\ \\  
  \texttt{\indent Pick a starting basis $\{v^{(0)}_1,\ldots,v^{(0)}_n$\} of $\mathbb{R}^n$.} \\
  \texttt{\indent Build the matrix $V = \left[v_1^{(0)}\right|\cdots \left| v_n^{(0)} \right]$} \\
  \texttt{\indent Obtain the factors $Q^{(0)}R^{(0)}=V^{(0)}$} \\
	\texttt{\indent For $k =1, 2, \ldots $ }\\
	\texttt{\indent \indent Let $W = AQ^{(k-1)}$ }\\
	\texttt{\indent \indent Obtain the factors $Q^{(k)}R^{(k)}=W$} \\

It is a fact that if the matrix $A$ has $n$ orthogonal eigenvectors $q_1, \ldots, q_n$ with corresponding real eigenvalues $|\lambda_1| > \ldots > |\lambda_n|$, and if the leading principal submatrices of the product $\left[ q_1 | \cdots | q_n \right]^TV^{(0)}$ are nonsingular, then the columns of $Q^{(k)}$ will converge towards a basis of eigenvectors of $A$. (Recall that the leading principal submatrices of the matrix $B$ are the top left square submatrices of $B$.) So, at last, here is a method that, under some hypotheses, computes all of the eigenvectors of the matrix $A$ at once.

However, this method is generally not used in practice. It turns out that there is a more elegant form of this algorithm, which we examine now.

\subsection{The QR Method}
The QR method for computing eigenvalues and eigenvectors \cite{QRAlgo}, like the simultaneous iteration method, allows the computation of all eigenvalues and eigenvectors of a real, symmetric, full rank matrix at once. Based upon the matrix decomposition from which it earned its name, the simplest form of the QR iteration algorithm may be written as follows:\\ \\
  \texttt{\indent Let $A^{(0)} = A$ }\\
	\texttt{\indent For $k =1, 2, \ldots $ }\\
	\texttt{\indent \indent Obtain the factors $Q^{(k)}R^{(k)}=A^{(k-1)}$} \\
	\texttt{\indent \indent Let $A^{(k)} = R^{(k)}Q^{(k)}$} \\

Simply put, in each iteration, we take the QR decomposition of the current matrix $A^{(k-1)}$, and multiply the factors $Q$ and $R$ in the \emph{reverse} order to obtain the new matrix $A^{(k)}$. 

It is surprising that this non-intuitive process would converge to anything useful, let alone a full set of eigenvectors and eigenvalues of $A$. However, it turns out that this algorithm can, in some sense, be seen to be equivalent to the simultaneous iteration method. The simultaneous iteration method of the previous section can be written as follows: \\ \\  
  \texttt{\indent Let $\underline{Q}^{(0)}=I$} \\
  	\texttt{\indent For $k =1, 2, \ldots $ }\\
	\texttt{\indent \indent Let $W = A\underline{Q}^{(k-1)}$ }\\
	\texttt{\indent \indent Obtain the factors $\underline{Q}^{(k)}R^{(k)}=W$} \\
	\texttt{\indent \indent Let $A^{(k)} = \left(\underline{Q}^{(k)}\right)^TA \underline{Q}^{(k)}$ } \\
	\texttt{\indent \indent Let $\underline{R}^{(k)}=R^{(k)}R^{(k-1)}\cdots R^{(1)}$} \\

Here we have renamed the $Q^{(k)}$ matrices of the previous section as $\underline{Q}^{(k)}$, and have introduced the matrices $A^{(k)}$ and $\underline{R}^{(k)}$ which do not affect the correctness of the algorithm.

The QR method of this section can be rewritten in the following way: \\ \\
	\texttt{\indent Let $A^{(0)} = A$ }\\
	\texttt{\indent For $k =1, 2, \ldots $ }\\
	\texttt{\indent \indent Obtain the factors $Q^{(k)}R^{(k)}=A^{(k-1)}$} \\
	\texttt{\indent \indent Let $A^{(k)} = R^{(k)}Q^{(k)}$} \\
	\texttt{\indent \indent Let $\underline{Q}^{(k)} = Q^{(1)}Q^{(2)} \cdots Q^{(k)}$} \\
	\texttt{\indent \indent Let $\underline{R}^{(k)} = R^{(k)}R^{(k-1)} \cdots R^{(1)}$} \\

Again, the introduction of the matrices $\underline{Q}^{(k)}$ and $\underline{R}^{(k)}$ do not affect the outcome of the algorithm. Note that by this algorithm, the following identities hold for all values of $k$:
\begin{itemize}
	\item $A^{(k-1)} = Q^{(k)}R^{(k)}$
	\item $A^{(k)} = R^{(k)}Q^{(k)}$	
	\item $\left(Q^{(k)}\right)^TQ^{(k)} = I$
\end{itemize}

The similarity between these two algorithms becomes apparent by the following theorem.

\begin{theorem}
The simultaneous iteration method and the QR method both generate the same sequences of matrices $A^{(k)}$, $\underline{Q}^{(k)}$, and $\underline{R}^{(k)}$, which satisfy the following relations:
\begin{enumerate}
\item \label{cond1} $A^k = \underline{Q}^{(k)} \underline{R}^{(k)}$ \\
\item \label{cond2} $A^{(k)} = \left(\underline{Q}^{(k)}\right)^TA \underline{Q}^{(k)}$
\end{enumerate}
\end{theorem} 
\begin{proof}
We proceed by induction on $k$.

When $k = 0$, it is clear that $A^{(0)}$, $\underline{Q}^{(0)}$, and $\underline{R}^{(0)}$ are the same for both the simultaneous iteration method algorithm and the QR method algorithm, and these values satisfy (\ref{cond1}) and (\ref{cond2}).

Suppose now that the values of these matrices are the same for both algorithms for some iteration $k - 1$, and that they satisfy the two properties (\ref{cond1}) and (\ref{cond2}) in this iteration.

In the simultaneous iteration method, we have 
\begin{eqnarray*}
A^{k} &=& AA^{k-1} \\
&=& A \left(\underline{Q}^{(k-1)} \underline{R}^{(k-1)}\right) \\
&=& \left( \underline{Q}^{(k)} R^{(k)} \right) \underline{R}^{(k-1)} \\
&=& \underline{Q}^{(k)} \underline{R}^{(k)}
\end{eqnarray*} 
and thus property (\ref{cond1}) is satisfied on the $k$th iteration; property (\ref{cond2}) is satisfied directly by definition of the simultaneous iteration algorithm.

In the QR method, we have
\begin{eqnarray*}
	A^{k} & = & AA^{k-1} \\
	& = & A\left(Q^{(1)}Q^{(2)}Q^{(3)}\ldots Q^{(k-1)}R^{(k-1)}\ldots R^{(1)}\right) \\
	& = & \left(Q^{(1)}R^{(1)}\right)\left(Q^{(1)}Q^{(2)}Q^{(3)}\ldots Q^{(k-1)}R^{(k-1)}\ldots R^{(1)}\right) \\
	& = & Q^{(1)}\left(Q^{(2)}R^{(2)}\right)Q^{(2)}Q^{(3)}\ldots Q^{(k-1)}R^{(k-1)}\ldots R^{(1)} \\
	& = & Q^{(1)}Q^{(2)}\left(Q^{(3)}R^{(3)}\right)Q^{(3)}\ldots Q^{(k-1)}R^{(k-1)}\ldots R^{(1)} \\
	& = & \ldots \\
	& = & Q^{(1)}Q^{(2)}Q^{(3)}\ldots Q^{(k-1)}\left(Q^{(k)}R^{(k)}\right)R^{(k-1)}\ldots R^{(1)} \\
	& = & \underline{Q}^{(k)} \underline{R}^{(k)}
\end{eqnarray*}
which proves that property (\ref{cond1}) holds for the $k$th iteration of the QR method.

We also have that
\begin{eqnarray*}
	A^{(k)} & = & R^{(k)}Q^{(k)} \\
	& = & \left(\left(Q^{(k)}\right)^TQ^{(k)}\right)R^{(k)}Q^{(k)} \\
	& = & \left(Q^{(k)}\right)^TA^{(k-1)}Q^{(k)} \\
	& = & \left(Q^{(k)}\right)^T\left(\left(\underline{Q}^{(k-1)}\right)^T A\underline{Q}^{(k-1)}\right)Q^{(k)} \\
	& = & \left(\underline{Q}^{(k)}\right)^TA\underline{Q}^{(k)}
\end{eqnarray*}
which proves that property (\ref{cond2}) holds for the $k$th iteration of the QR method as well.

By hypothesis, the values of $A^{(k-1)}$, $\underline{Q}^{(k-1)}$, and $\underline{R}^{(k-1)}$ were the same for both algorithms. Since both algorithms satisfy (\ref{cond1}) on the $k$th iteration, we have $A^k = \underline{Q}^{(k)} \underline{R}^{(k)}$ for both algorithms, and since the QR decomposition is unique, it follows that $\underline{Q}^{(k)}$ and $\underline{R}^{(k)}$ are also the same for both algorithms. Both algorithms also satisfy (\ref{cond2}) on the $k$th iteration, which means that the matrix $A^{(k)} = (\underline{Q}^{(k)})^TA\underline{Q}^{(k)}$ is also the same for both algorithms on the $k$th iteration. 

Hence, both algorithms produce the same values for the matrices $A^{(k)}$, $\underline{Q}^{(k)}$, and $\underline{R}^{(k)}$ which satisfy the relationships (\ref{cond1}) and (\ref{cond2}). By induction, this holds for all $k$, proving the theorem.
\end{proof}

Now, since we saw that the columns of $\underline{Q}^{(k)}$ converged to a basis of eigenvectors of $A$ in the simultaneous iteration method, the result of the above theorem tells us that the same holds for the columns of $\underline{Q}^{(k)}$ as computed using the QR algorithm. In particular, the columns $\underline{q}_i$ of $\underline{Q}^{(k)}$ each converge to a unit eigenvector $q_i$ of $A$.

Property (\ref{cond2}) of the theorem tells us that $$A_{ij}^{(k)} = \left(\underline{q}_i^{(k)}\right)^T A\underline{q}_j^{(k)}$$ where $\underline{q}_i^{(k)}$ and $\underline{q}_j^{(k)}$ are columns $i$ and $j$, respectively, of $\underline{Q}^{(k)}$. But, we saw that $\underline{q}_i^{(k)} \rightarrow q_i$ and $\underline{q}_j^{(k)} \rightarrow q_j$ as $k \rightarrow \infty$, where $q_i$ and $q_j$ are unit eigenvectors of $A$ and are orthogonal if $i \neq j$.

In the case that $i \neq j$, we have that $$A_{ij}^{(k)} \rightarrow q_i^{T} Aq_j = \lambda_j q_i^{T}q_j = 0$$ since the eigenvectors are orthogonal; thus $A^{(k)}$ approaches a diagonal matrix.

In the case that $i = j$, we have that $$A_{ii}^{(k)} \rightarrow q_i^{T} Aq_i = \lambda_i q_i^{T}q_i = \lambda_i$$ since the eigenvectors are unit vectors; thus the diagonal elements of $A^{(k)}$ approach the eigenvalues of $A$.

We finally see that the QR iteration method provides an orthogonal matrix $\underline{Q}^{(k)}$ whose columns approach eigenvectors of $A$, and a diagonal matrix $A^{(k)}$ whose diagonal elements approach the eigenvalues of $A$. Thus the QR method presents a simple, elegant algorithm for finding the eigenvectors and eigenvalues of a real, symmetric, full rank matrix. 

The QR method is, in essence, the same as the simultaneous iteration method, and therefore suffers the same restriction on the matrices it can be applied to. Unfortunately, these two methods may fail when there are non-real eigenvalues, or when the eigenvectors do not form an orthogonal basis of $\mathbb{R}^n$. However, symmetric matrices with real entries do occur frequently in many physical applications, so there is some use for these algorithms.

\section{Conclusions and Improvements}

The purpose of this paper was to provide an overview of some of the numeric techniques used to compute eigenvalues and eigenvectors of matrices. These methods are all based on simple ideas that were steadily generalized and adapted until they became powerful iterative algorithms. 

However, even the QR iteration method as presented in this paper is generally not suitable for use in practice, even in the situations when it can be applied. There are many ways to refine the  algorithms we have seen in order to speed up the implementations. For example, before applying the QR algorithm, it is common to reduce the matrix $A$ into a simpler form, such as a tridiagonal matrix: performing QR decompositions on the new matrix then becomes much faster. It is also possible to improve the QR iteration method by incorporating shifts and Rayleigh quotients, just as these concepts helped improve the original power iteration method. The variations that are implemented in practice are vastly more sophisticated than the simple algorithms presented here.

As with all algorithms, the question of improvement is always an open problem. Modern information processing deals with increasingly large matrices, and efficient methods for computing eigenvalues and eigenvectors are extremely necessary. The techniques of the future may become much more complicated in order to keep up with this growing demand, but as we have seen, there is a lot that can be done using simple ideas and well-chosen generalizations.

\section{Acknowledgements}

I would like to thank the referees for their careful reviewing and thorough feedback on this paper. I would also like to thank Dr.~Pascal Poupart for introducing me to the study of computational linear algebra.


\bibliographystyle{abbrv}
\bibliography{NumericalEigenstuffBib}

\end{document}